\documentclass{amsart}
\usepackage{amsmath, amssymb, mathrsfs, amsthm, MnSymbol, xifthen, verbatim, color, mathrsfs}

\usepackage[arrow, matrix]{xy}

\definecolor{DarkBlue}{rgb}{0,0.2,0.6}
\definecolor{PinkPurple}{rgb}{0.8,0.3,0.3}

\usepackage[pdftex,
            pdfauthor={Mehdi Ghasemi, Salma Kuhlmann, Murray Marshall},
            pdftitle={Application of Jacobi's Theorem to lmc Real Algebras},
            pdfsubject={LMC Algebras},
            pdfkeywords={Positivity, Sums of 2d-powers, LMC Topologies, Norms, Seminorms, Continuous Functionals, Moment Problem},
            pdfproducer={Latex with hyperref},
            pdfcreator={PDFLaTeX},
            linkcolor=DarkBlue,
            citecolor=PinkPurple,
            colorlinks=true]{hyperref}

\newtheorem{thm}{Theorem}[section]
\newtheorem{lemma}[thm]{Lemma}

\newtheorem{crl}[thm]{Corollary}

\theoremstyle{definition}
\newtheorem{dfn}[thm]{Definition}

\newtheorem{rem}[thm]{Remark}

\newcommand{\reals}{\mathbb{R}}

\newcommand{\cplx}{\mathbb{C}}

\newcommand{\rx}{\sgr{\mathbb{R}}{\ux}}
\newcommand{\sos}{\sum\mathbb{R}[\underline{X}]^2}

\newcommand{\ux}{\underline{X}}

\newcommand{\supp}{\text{supp}}

\newcommand{\K}[1]{\mathcal{K}_{#1}}
\newcommand{\pos}[1]{\mbox{Pos}(#1)}

\newcommand{\V}[1]{\mathcal{X}(#1)}

\newcommand{\M}[1]{\mathfrak{sp}(#1)}

\newcommand{\sgr}[2]{#1[#2]}
\newcommand{\ringsop}[2]{\sum #1^{#2}}

\newcommand{\norm}[2]{\|\ifthenelse{\isempty{#2}}{\cdot}{#2}\|_{#1}}
\newcommand{\cl}[2]{\overline{#2}^{\ifthenelse{\isempty{#1}}{}{#1}}}
\newcommand{\Cnt}[2]{\mathrm{C}_{\ifthenelse{\isempty{#1}}{}{#1}}(#2)}
\newcommand{\Psd}[2]{\mbox{Psd}_{\ifthenelse{\isempty{#1}}{}{#1}}(#2)}

\newcommand{\map}[3]{#1:#2\longrightarrow #3}

\begin{document}

\title[Application of Jacobi's Theorem to lmc $\reals$-Algebras]{Application of Jacobi's Representation Theorem to locally multiplicatively convex topological $\reals$-Algebras}
\author[M. Ghasemi, S. Kuhlmann, M. Marshall]{Mehdi Ghasemi$^1$, Salma Kuhlmann$^2$, Murray Marshall$^1$}
\address{$^1$Department of Mathematics and Statistics,\newline\indent
University of Saskatchewan,\newline\indent
Saskatoon, SK. S7N 5E6, Canada}
\email{mehdi.ghasemi@usask.ca, marshall@math.usask.ca}
\address{$^2$Fachbereich Mathematik und Statistik,\newline\indent
Universit\"{a}t Konstanz\newline\indent
78457 Konstanz, Germany}
\email{salma.kuhlmann@uni-konstanz.de}
\keywords{positivity, sums of squares, sums of 2d-powers,
locally multiplicatively convex topologies, norms and seminorms,
continuous linear functionals, moment problem}

\subjclass[2010]{Primary 13J30, 14P99,
44A60; Secondary 43A35, 46B99, 44A60.}

\date{\today}

\begin{abstract}
Let $A$ be a commutative unital $\reals$-algebra and let $\rho$ be a seminorm on $A$ which satisfies $\rho(ab)\leq\rho(a)\rho(b)$. We apply T. Jacobi's
representation theorem \cite{J} to determine the closure of a $\ringsop{A}{2d}$-module $S$ of $A$ in the topology induced by $\rho$, for any integer
$d\ge1$. We show that this closure is exactly the set of all elements $a\in A$ such that $\alpha(a)\ge0$ for every $\rho$-continuous $\reals$-algebra homomorphism $\map{\alpha}{A}{\reals}$ with $\alpha(S)\subseteq[0,\infty)$, and that this result continues to hold when $\rho$ is replaced by any locally multiplicatively convex topology $\tau$ on $A$. We obtain a representation of any linear functional $L : A \rightarrow \reals$ which is continuous with respect to any such  $\rho$ or $\tau$ and non-negative on $S$ as integration with respect to a unique Radon measure on the space of all real valued $\reals$-algebra homomorphisms on $A$, and we characterize the support of the measure obtained in this way.
\end{abstract}

\maketitle


\section{Introduction}
It was known to Hilbert \cite{Hil} that a nonnegative real multivariable polynomial
$f=\sum_{\alpha}f_{\alpha}\ux^{\alpha}\in\rx:=\reals[X_1,\dots,X_n]$
is not necessarily a sum of squares of polynomials. However, every such polynomial can be approximated by elements of the cone $\sos:=$ sums of
squares of polynomials, with respect to the topology induced by the norm $\norm{1}{}$
(given by $\norm{1}{\sum_{\alpha}f_{\alpha}\ux^{\alpha}}:=\sum_{\alpha}|f_{\alpha}|$).
In fact, every polynomial $f\in\rx$, nonnegative on $[-1,1]^n$ is in the $\norm{1}{}$-closure of $\sos$ \cite[Theorem 9.1]{BCR}.
Moreover, it is known that for every $f\in\pos{[-1,1]^n}:=$ the cone of nonnegative polynomials on $[-1,1]^n$, and $\epsilon>0$, there exists $N>0$
depending on $n$, $\epsilon$, $\deg f$ and the size of coefficients of $f$ such that for every integer $r\ge N$, the polynomial
$f_{\epsilon,r}:=f+\epsilon(1+\sum_{i=1}^nX_i^{2r})\in\sos$.
This gives an effective way of approximating $f$ by sums of squares in $\norm{1}{}$ \cite[Theorem 3.9]{JLN}.
The closure of $\sos$ with respect to the family of weighted $\norm{p}{}$-norms has been studied in \cite{GKS}.
Note that an easy application of Stone-Weierstrass Theorem shows that the same result holds for the coarser norm
$\norm{\infty}{f}:=\sup_{x\in[-1,1]^n}|f(x)|$; i.e., $\cl{\norm{\infty}{}}{\sos}=\pos{[-1,1]^n}$, but in practice, finding $\norm{\infty}{f}$ is a
computationally difficult optimization problem, whereas $\norm{1}{f}$ is easy to compute. Therefore to gain more computational flexibility it is
interesting to study such closures with respect to various norms on $\rx$.

The general set-up we consider is the following. Let $C$ be a cone in $\rx$, $\tau$ a locally convex topology on $\rx$
and $K\subseteq\reals^n$ be a closed set. Consider the condition:
\begin{equation}\label{ConeClEq}
\cl{\tau}{C}\supseteq \pos{K},
\end{equation}
(where as above, $\pos{K}$ denotes the set of polynomials nonnegative on $K$).
An application of Hahn-Banach Separation Theorem together with Haviland's Theorem
(see Theorem \ref{Haviland}) shows that \textit{\eqref{ConeClEq} holds if and only if for every $\tau$-continuous linear functional $L$ with
$L(C)\subseteq[0,\infty)$, there exists a Borel measure $\mu$ on $K$ such that}
\begin{equation}\label{LinRep}
	\forall f\in\rx\quad L(f)=\int_K f~d\mu.
\end{equation}

In the present paper, we study closure results of type \eqref{ConeClEq} and their corresponding representation results of type \eqref{LinRep}
for \textit{any} locally multiplicatively convex (unital, commutative) topological $\reals$-algebra.

In Section \ref{prelim} we introduce some terminology and notation and recall Jacobi's Theorem and a generalized version of Haviland's Theorem,
results which play a crucial role throughout the paper.

In Section \ref{snAlgs} we consider the case of a submultiplicative seminorm $\rho$ on an $\reals$-algebra $A$.
In Theorem \ref{MainThm} we prove that for any integer $d\ge1$ and any $\ringsop{A}{2d}$-module $S$ of $A$, $\cl{\rho}{S}$ consists of all elements
of $A$ with nonnegative image under every $\rho$-continuous $\reals$-algebra homomorphism $\map{\alpha}{A}{\reals}$ such that
$\alpha(S)\subseteq[0,\infty)$. This generalizes \cite[Theorem 5.3]{G-K} on the closure of $\ringsop{A}{2d}$ with respect to a submultiplicative
norm.
The application of Theorem \ref{MainThm} to the representation of linear functionals by measures is explained in
Corollary \ref{GenMoment}.

In Section \ref{seminormed $*$-algebras} we explain how Theorem \ref{MainThm} and Corollary \ref{GenMoment} apply in the case of a (unital,
commutative) $*$-algebra equipped with a submultiplicative $*$-seminorm. Corollary \ref{*-algebra version} generalizes results on $*$-semigroup
algebras in \cite[Theorem 4.2.5]{BCRBK} and \cite[Theorem 4.3 and Corollary 4.4]{GMW}.

In Section \ref{lmcalgs}, specifically in Theorem \ref{MainThm2}, we explain how 
Theorem \ref{MainThm} extends to the class
of locally multiplicatively convex topologies.
Such topologies are induced by families of submultiplicative seminorms.
Theorem \ref{MainThm2} can viewed as a strengthening (in the commutative case) of the result in \cite[Lemma 6.1 and
Proposition 6.2]{Schm2} about enveloping algebras of Lie algebras.

\section{Preliminaries}\label{prelim}
Throughout $A$ denotes a unitary commutative $\reals$-algebra.
The set of all unitary $\reals$-algebra homomorphisms from $A$ to $\reals$ will be denoted by $\V{A}$. Note that $\V{A}$ as a subset of $\reals^A$ carries a natural topology, where
$\reals^A$ is endowed with the product topology. This topology coincides with the weakest topology on $\V{A}$ which makes all the evaluation maps
$\map{\hat{a}}{\V{A}}{\reals}$, defined by $\hat{a}(\alpha)=\alpha(a)$ continuous \cite[section 5.7]{MPS}.

For an integer $d\ge1$, $\ringsop{A}{2d}$ denotes the set of all finite sums of $2d$ powers of elements of $A$.
A \textit{$\ringsop{A}{2d}$-module} of $A$ is a subset $S$ of $A$ such that $1\in S$, $S+S\subseteq S$ and $a^{2d}\cdot S\subseteq S$ for each $a\in A$.
We say $S$ is \textit{archimedean} if for each $a\in A$ there exists an integer $n\ge1$ such that $n+ a\in S$.
For any subset $S$ of $A$, the non-negativity set of $S$, denoted by $\K{S}$, is defined by
\[
	\K{S}:= \{ \alpha \in \V{A} : \hat{a}(\alpha)\ge 0 \text{ for all } a\in S\}.
\]
Also, for $K\subseteq \V{A}$, we define $\pos{K}$ by
\[
	\pos{K}:=\{a\in A : \hat{a}(\alpha) \ge 0 \text{ for all } \alpha \in K\}.
\]

\begin{thm}[Jacobi]\label{Jacobi} Suppose $S$ is an archimedean $\ringsop{A}{2d}$-module of $A$ for some integer $d\ge1$. Then for each $a\in A$,
\[
	\hat{a}>0 ~ on ~ \K{S}\Rightarrow a\in S.
\]
\end{thm}
\begin{proof}
See \cite[Theorem 4]{J}.
\end{proof}

Recall that a \textit{Radon measure} on a Hausdorff topological space $X$ is a measure on the $\sigma$-algebra of Borel sets of $X$ that is locally finite and inner regular.  \textit{Locally finite} means that every point has a neighbourhood of finite measure.  \textit{Inner regular} means each Borel set can be approximated from within using a compact set.
We will use the following version of Haviland's Theorem to get representations of linear functionals on $A$.
\begin{thm}\label{Haviland}
Suppose $A$ is an $\reals$-algebra, $X$ is a  Hausdorff space, and $\map{\hat{~}}{A}{\Cnt{}{X}}$ is an $\reals$-algebra homomorphism such that for some $p\in A$,
$\hat{p}\ge0$ on $X$, the set $X_i=\hat{p}^{-1}([0,i])$ is compact for each $i=1,2,\cdots$. Then for every linear functional $\map{L}{A}{\reals}$ satisfying
\[
    L(\{a\in A:\hat{a}\ge0\textrm{ on }X\})\subseteq[0,\infty),
\]
there exists a Radon measure $\mu$ on $X$ such that $\forall a\in A \quad L(a)=\int_X\hat{a}~d\mu$.
\end{thm}
Here, $\Cnt{}{X}$ denotes the ring of all continuous real valued functions on $X$.
A proof of Theorem \ref{Haviland} can be found in  \cite[Theorem 3.1]{Marshall} or \cite[Theorem 3.2.2]{MPS} (also see \cite{Hav1, Hav2} for the original version). Note that the hypothesis of Theorem \ref{Haviland} implies in particular that $X$ is locally compact (so $\mu$ is actually a Borel measure).

\section{Seminormed $\reals$-Algebras}\label{snAlgs}

\begin{dfn}
A \textit{seminorm} $\rho$ on $A$ is a map $\map{\rho}{A}{[0,\infty)}$ such that
\newline
\indent	(1) for $x\in A$ and $r\in\reals$, $\rho(rx)=|r|\rho(x)$, and\newline
\indent (2) for all $x,y\in A$, $\rho(x+y)\leq\rho(x)+\rho(y)$.\newline
Moreover, $\rho$ is called a \textit{submultiplicative} seminorm if in addition:\newline
\indent (3) for all $x,y\in A$, $\rho(xy)\leq\rho(x)\rho(y)$.
\end{dfn}
The algebra $A$ together with a submultiplicative seminorm $\rho$ on $A$ is called a \textit{seminormed algebra} and is denoted by the symbolism $(A,\rho)$.
We denote the set of all $\rho$-continuous $\reals$-algebra homomorphisms from $A$ to $\reals$ by $\M{\rho}$,
which we refer to as the
\textit{Gelfand spectrum} of $(A,\rho)$. The topology on $\M{\rho}$ 
is the topology induced as a subspace of $\V{A}$.
\begin{lemma}\label{bddmrph}
For any submultiplicative seminorm $\rho$ on $A$, $$\M{\rho} = \{ \alpha \in \V{A} : |\alpha(x)|\leq\rho(x) \text{ for all } x\in A\}.$$
\end{lemma}
\begin{proof} Suppose $\alpha \in \V{A}$ and there exists $x\in A$ such that $|\alpha(x)|>\rho(x)$. Set $y= \frac{x}{\delta}$ where $\delta \in \reals$ is such that $|\alpha(x)|>\delta>\rho(x)$. Then $\rho(y)<1$ and $|\alpha(y)|>1$ so, as $n\rightarrow \infty$, $\rho(y^n) \rightarrow 0$ and $|\alpha(y^n)|\rightarrow \infty$. This proves ($\subseteq$). The other inclusion is clear.
\end{proof}

\begin{crl}\label{compact} For any submultiplicative seminorm $\rho$ on $A$, $\M{\rho}$
is compact.
\end{crl}

\begin{proof} The map
$\alpha\mapsto(\hat{a}(\alpha))_{a\in A}$ identifies $\M{\rho}$ with a closed subset of the compact space $\prod_{a\in A}[-\rho(a),\rho(a)]$.
\end{proof}
\begin{rem}\label{CpltnSNAlg}
For a seminormed algebra $(A,\rho)$, the set $I_{\rho}:=\{a\in A : \rho(a)=0\}$ is a closed ideal of $A$ and the map
\[
	\map{\bar{\rho}}{\bar{A}=A/I_{\rho}}{[0,\infty)}
\]
defined by $\bar{\rho}(\bar{a})=\rho(a)$ is a well-defined norm on $\bar{A}$. Thus $(\bar{A},\bar{\rho})$ is a normed $\reals$-algebra and hence $(\bar{A},\bar{\rho})$ admits
a completion $(\tilde{A},\tilde{\rho})$ which is a Banach $\reals$-algebra.
\end{rem}

\begin{lemma}\label{SemiNormSo2d} For any unital Banach $\reals$-algebra $(B,\varphi)$, any $a\in A$ and $r\in \reals$ such that $r>\varphi(a)$, and any integer $k \ge 1$, there exists $p\in B$ such that $p^k=r+a$.
\end{lemma}
\begin{proof} This is well-known. The standard power series expansion $$(r+x)^{1/k} = r^{1/k}(1+\frac{x}{r})^{1/k} = r^{1/k}\sum_{i=0}^{\infty} \frac{\frac{1}{k}(\frac{1}{k}-1)\dots (\frac{1}{k}-i)}{i!}(\frac{x}{r})^i$$ converges absolutely for $|x|<r$. This implies that $$p := r^{1/k}\sum_{i=0}^{\infty} \frac{\frac{1}{k}(\frac{1}{k}-1)\dots (\frac{1}{k}-i)}{i!}(\frac{a}{r})^i$$ is a well-defined element of $B$ and $p^k=r+a$.
\end{proof}

\begin{crl} \label{automatic continuity} For any unital Banach $\reals$-algebra $(B,\varphi)$ and any linear functional $L: B \rightarrow \reals$, if $L(b^{2d})\ge 0$ for all $b\in B$ for some $d\ge 1$ then $L$ is $\varphi$-continuous. In particular, each $\alpha \in \V{B}$ is $\varphi$-continuous.
\end{crl}

\begin{proof} By Lemma \ref{SemiNormSo2d}, $\frac{1}{n}+\varphi(a)\pm a = \frac{1}{n}+\varphi(\pm a)+(\pm a) \in B^{2d}$ for all $a\in B$ and all $n\ge 1$. Applying $L$ this yields $|L(a)| \le (\frac{1}{n}+\varphi(a))L(1)$ for all $a\in B$ and all $n\ge 1$  so $|L(a)| \le \varphi(a)L(1)$ for all $a\in B$.
\end{proof}

We come now to the main result of the section.

\begin{thm}\label{MainThm}
Let $\rho$ be a submultiplicative seminorm on $A$ and let $S$ be a $\sum A^{2d}$-module of $A$. Then $\cl{\rho}{S}=\pos{\K{S}\cap \M{\rho}}$. In particular, $\cl{\rho}{\ringsop{A}{2d}}=\pos{\M{\rho}}$.
\end{thm}
\begin{proof}
Since each $\alpha\in\K{S}\cap \M{\rho}$ is continuous and $$\pos{\K{S}\cap \M{\rho}}=\bigcap_{\alpha\in\K{S}\cap \M{\rho}}\alpha^{-1}([0,\infty)),$$ $\pos{\K{S}\cap \M{\rho}}$ is
$\rho$-closed. Since $S\subseteq\pos{\K{S}\cap \M{\rho}}$ this implies $\cl{\rho}{S}\subseteq\pos{\K{S}\cap \M{\rho}}$.
For the reverse inclusion we have to show that if $b\in\pos{\K{S}\cap \M{\rho}}$ then $b \in \cl{\rho}{S}$. Let $\tilde{S}$ denote the closure of the image of $S$ in $(\tilde{A}, \tilde{\rho})$. Then $\tilde{S}$ is a $\sum \tilde{A}^{2d}$-module of $\tilde{A}$. By Lemma \ref{SemiNormSo2d}, $\frac{1}{n}+\tilde{\rho}(a)+ a \in \tilde{A}^{2d} \subseteq \tilde{S}$ for all $a\in \tilde{A}$ and all $n\ge 1$, so $\tilde{\rho}(a) + a \in \tilde{S}$ for all $a\in \tilde{A}$. This implies that $\tilde{S}$ is archimedean.
By Corollary \ref{automatic continuity} every element of $\K{\tilde{S}}$ restricts to an element of $\K{S}\cap \M{\rho}$\footnote{In fact one can show that the restriction map  $\K{\tilde{S}} \rightarrow \K{S}\cap \M{\rho}$ is a homeomorphism.} so, by our hypothesis on $b$, $\alpha(\tilde{b}) =\alpha|_A(b) \ge 0$ for all $\alpha \in \K{\tilde{S}}$, where $\tilde{b}$ denotes the image of $b$ in $\tilde{A}$. Then $\alpha(\tilde{b}+\frac{1}{n}) >0$ for all $\alpha \in \K{\tilde{S}}$ so, by Jacobi's Theorem \ref{Jacobi}, $\tilde{b}+\frac{1}{n} \in \tilde{S}$ for all $n\ge 1$. Then $\tilde{b} \in \tilde{S}$, so $b\in \cl{\rho}{S}$.
\end{proof}

\begin{crl}\label{GenMoment}
Let $\rho$ be a submultiplicative seminorm on $A$,  $S$ a $\sum A^{2d}$-module of $A$.
If $\map{L}{A}{\reals}$ is a $\rho$-continuous linear functional such that $L(s)\ge 0$ for all $s\in S$ then there exists a unique Radon measure $\mu$ on $\V{A}$ such that
\[
	\forall a\in A \quad L(a)=\int \hat{a}~d\mu.
\]
Moreover, $\supp(\mu) \subseteq \K{S}\cap \M{\rho}$.
\end{crl}
\begin{proof}
By our hypothesis and Theorem \ref{MainThm} $L$ is non-negative on $\pos{\mathcal{K}_{S}\cap \M{\rho}}$. Applying Theorem \ref{Haviland}, with $X:=\K{S}\cap \M{\rho}$ and $\hat{} : A \rightarrow \Cnt{}{X}$ the map defined by $a \mapsto \hat{a}|_X$, yields a Radon measure $\mu'$ on $X$ such that $L(a) = \int_X \hat{a} d\mu'$ for all $a\in A$. Observe that $X$ is compact, by Corollary \ref{compact}, so we can take $p=1$. The Radon measure $\mu$ on $\V{A}$ that we are looking for is just the extension of $\mu'$ to $\V{A}$, i.e., $\mu(E) := \mu'(E\cap X)$ for all Borel sets $E$ in $\V{A}$. Uniqueness of $\mu$ is a consequence of the following easy result.
\end{proof}

\begin{lemma} Suppose $\mu$ is a Radon measure on $\V{A}$ having compact support. Then $\mu$ is determinate, i.e., if $\nu$ is any Radon measure on $\V{A}$ satisfying $\int \hat{a} d\nu = \int \hat{a} d\mu$ for all $a\in A$ then $\nu = \mu$.
\end{lemma}

\begin{proof} Set $Y = \supp(\mu)$. Suppose first that $\supp(\nu) \not\subseteq Y$. Then there exists a compact set $Z \subseteq \V{A} \backslash Y$ with $\nu(Z)>0$. Choose $\epsilon >0$ so that $\epsilon <\frac{\nu(Z)}{\mu(Y)+\nu(Z)}$.  Since $Y,Z$ are compact and disjoint, the Stone-Weierstrass Theorem implies  there exists $a\in A$ such that $|\hat{a}(\alpha)| \le \epsilon$ for all $\alpha \in Y$ and $|\hat{a}(\alpha)-1| \le \epsilon$ for all $\alpha \in Z$. Replacing $a$ by $a^2$ if necessary, we can suppose $\hat{a}\ge 0$ on $\V{A}$.  Then $\int \hat{a}d\mu \le \epsilon \mu(Y)$, but $\int \hat{a} d\nu \ge \int_Z \hat{a} d\nu \ge (1-\epsilon)\nu(Z)$, which is a contradiction. It follows that $\supp(\nu) \subseteq Y$, so $\mu,\nu$ both have support in the same compact set $Y$. Then, using the Stone-Weierstrass Theorem again, $\int \varphi d\mu = \int \varphi d\nu$ for all $\varphi \in \Cnt{}{Y}$ so $\mu = \nu$, by the Riesz Representation Theorem.
\end{proof}

\begin{rem} (i) The converse of Corollary \ref{GenMoment} holds trivially: If $L(a) = \int \hat{a} d \mu$ for all $a\in A$ for some Radon measure $\mu$ with $\supp(\mu) \subseteq \K{S}\cap \M{\rho}$ then $L(s)\ge 0$ for all $s\in S$ and $|L(a)| \le \rho(a)L(1)$ for all $a\in A$, so $L$ is $\rho$-continuous.

(ii) Theorem \ref{MainThm} and Corollary \ref{GenMoment} should be viewed as `two sides of the same coin'. We have shown how Corollary \ref{GenMoment} can be deduced from Theorem \ref{MainThm} using Theorem \ref{Haviland}. Conversely, one can deduce Theorem \ref{MainThm} from Corollary \ref{GenMoment} by an easy application of the Hahn-Banach Separation Theorem.
\end{rem}

\section{$*$-seminormed $*$-algebras}\label{seminormed $*$-algebras}

In this section we consider a  $*$-algebra $R$ equipped with a submultiplicative $*$-seminorm $\varphi$, i.e., $R$ is a (unital, commutative) $\cplx$-algebra equipped with an involution $* : R \rightarrow R$ satisfying $$(\lambda a)^* = \overline{\lambda}a^*, \ (a+b)^* = a^*+b^*, \ (ab)^* = a^*b^* \text{ and } a^{**} = a$$ for all $\lambda \in \cplx$ and all $a,b \in R$, and $\varphi : R \rightarrow [0,\infty)$ satisfies   $$\varphi(\lambda a) = |\lambda|\varphi(a), \ \varphi(a+b)\le \varphi(a)+\varphi(b), \ \varphi(ab)\le \varphi(a)\varphi(b) \text{ and } \varphi(a^*) = \varphi(a)$$ for all $\lambda \in \cplx$ and all $a,b\in R$.

We denote by $\V{R}$ set of all $*$-algebra homomorphisms $\alpha : R \rightarrow \cplx$ equipped with its natural topology as a subspace of the product space $\cplx^{R}$ and by $\M{\varphi}$ 
the subspace of $\V{R}$ consisting of all $\varphi$-continuous $*$-algebra homomorphisms $\alpha : R \rightarrow \cplx$. The \textit{symmetric part} of $R$ is $$H(R):= \{ a\in R : a^* = a\}.$$ Since $R= H(R)\oplus iH(R)$, one sees that $\V{R}$ and $\M{\varphi}$ are naturally identified via restriction with $\V{H(R)}$ and $\M{\varphi|_{H(R)}}$, respectively, and $\varphi$ continuous $*$-linear functionals $L : R \rightarrow \cplx$ are naturally identified via restriction with $\varphi|_{H(R)}$-continuous $\reals$-linear functionals $L : H(R) \rightarrow \reals$.

Applying Theorem \ref{MainThm} and Corollary \ref{GenMoment} to the symmetric part of $(R,\varphi)$ yields the following result.

\begin{crl} \label{*-algebra version} Let $R$ be a $*$-algebra equipped with a submultiplicative $*$-seminorm $\varphi$, $S$ a $\sum H(R)^{2d}$-module of $H(R)$. Then $\overline{S}^{\varphi} = \pos{\K{S}\cap \M{\varphi}}$. If $L : R\rightarrow \cplx$ is any $\varphi$-continuous $*$-linear functional such that $L(s)\ge 0$ for all $s\in S$ then there exists a unique Radon measure on $\V{R}$ such that $L(a) = \int \hat{a} d\mu$ for all $a\in R$. Moreover, $\supp(\mu) \subseteq \K{S}\cap \M{\varphi}$.
\end{crl}

Corollary \ref{*-algebra version} applies, in particular, to any  $*$-semigroup algebra $\cplx[W]$ equipped with a $*$-seminorm $\|\cdot\|_{\phi}$ arising from an absolute value $\phi$ on the $*$-semigroup $W$, i.e., $\| \sum \lambda_ww\|_{\phi} := \sum_w |\lambda_w|\phi(w)$. In this way Corollary \ref{*-algebra version} extends \cite[Theorem 4.2.5]{BCRBK} and \cite[Theorem 4.3 and Corollary 4.4]{GMW}.

\section{Locally Multiplicatively Convex Topologies}\label{lmcalgs}

Let $A$ be an $\reals$-algebra. A subset $U$ of $A$ is called a \textit{multiplicative set} (an \textit{m-set} for short) if $U\cdot U\subseteq U$. A locally convex vector space topology on $A$ is
said to be \textit{locally multiplicatively convex} (\textit{lmc} for short) if there exists a system of neighbourhoods for $0$ consisting of
\textit{m}-sets. It is immediate from the definition that multiplication is continuous in any lmc-topology. We recall the following result.

\begin{thm}\label{lmcSN}
 A locally convex vector space topology $\tau$ on $A$ is lmc if and only if $\tau$ is generated by a family of submultiplicative seminorms on $A$.
\end{thm}
\begin{proof}
See \cite[4.3-2]{B-N-S}.
\end{proof}

A family $\mathcal{F}$ of submultiplicative seminorms of $A$ is said to be \textit{saturated} if, for any $\rho_1, \rho_2 \in \mathcal{F}$, the seminorm $\rho$ of $A$ defined by
\[
	\rho(x):=\max\{ \rho_1(x),\rho_2(x)\} \text{ for all } x\in A
\]
belongs to $\mathcal{F}$.
For an lmc topology $\tau$ on $A$ one can always assume that the family $\mathcal{F}$ of submultiplicative seminorms generating $\tau$ is saturated.
In this situation the topology $\tau$ is the inductive limit topology, i.e., the balls $B_r^{\rho}(0) := \{ a \in A : \rho(a)<r\},$ $\rho \in \mathcal{F}$, $r>0$  form a system of $\tau$-neighbourhoods of zero. This is clear.

We record the following more-or-less obvious result:

\begin{lemma}\label{lmc2} Suppose $\tau$ is an lmc topology on $A$ generated by a saturated family $\mathcal{F}$ of submultiplicative seminorms of $A$ and $L: A \rightarrow \reals$ is a $\tau$-continuous linear functional. Then there exists 
$\rho \in \mathcal{F}$ such that $L$ is $\rho$-continuous.
\end{lemma}

\begin{proof} The set $\{ a \in A : |L(a)|<1\}$ is an open neighbourhood of $0$ in $A$ so there exists $\rho \in \mathcal{F}$ and $r>0$ such that $B_r^{\rho}(0) \subseteq \{ a \in A : |L(a)|<1\}$. Then $B_{r\epsilon}^{\rho}(0) = \epsilon B_r^{\rho}(0)$ so $$L(B_{r\epsilon}^{\rho}(0)) = L(\epsilon B_r^{\rho}(0)) = \epsilon L(B_r^{\rho}(0)) \subseteq \epsilon (-1,1) = (-\epsilon,\epsilon)$$ for all $\epsilon>0$, i.e., $L$ is $\rho$-continuous.
\end{proof}

We denote the Gelfand spectrum of $(A,\tau)$, i.e., the  set of all $\tau$-continuous $\alpha \in \V{A}$, 
by $\M{\tau}$ for short.

\begin{crl} Suppose $\tau$ is an lmc topology on $A$ generated by a saturated family $\mathcal{F}$ of submultiplicative seminorms of $A$. Then $\M{\tau} = \bigcup\limits_{\rho \in \mathcal{F}} \M{\rho}$.
\end{crl}

Our main result in the previous section extends to general lmc topologies, as follows:

\begin{thm}\label{MainThm2}
Let $\tau$ be an lmc topology on $A$  and let $S$ be any $\sum A^{2d}$-module of $A$. Then $\cl{\tau}{S}=\pos{\mathcal{K}_{S}\cap \M{\tau}}$. In particular, $\cl{\tau}{\ringsop{A}{2d}}=\pos{\M{\tau}}$.
\end{thm}

\begin{proof} Let $\mathcal{F}$ be a saturated family of submultiplicative seminorms generating $\tau$. Then $\cl{\tau}{S}= \bigcap\limits_{\rho \in \mathcal{F}}\cl{\rho}{S} = \bigcap\limits_{\rho \in \mathcal{F}} \pos{\mathcal{K}_{S}\cap \M{\rho}} = \pos{\mathcal{K}_{S}\cap \M{\tau}}$.
\end{proof}

In view of Lemma \ref{lmc2}, Corollary \ref{GenMoment} also extends to general lmc topologies in an obvious way. The unique Radon measure corresponding to a $\tau$-continuous linear functional $L : A \rightarrow \reals$ such that $L(s) \ge 0$ for all $s\in S$ has support contained in the compact set $\K{S}\cap \M{\rho}$ for some $\rho \in \mathcal{F}$.

The finest lmc topology on $A$ is the lmc topology generated by the family of all submultiplicative seminorms of $A$.
Theorem \ref{MainThm2} can thought of as a strengthening (in the commutative case) of the result of \cite[Lemma 6.1 and Proposition 6.2]{Schm2} about enveloping algebras for $\reals$-algebras. Note also the following:
\begin{crl}
Let $\eta$ be the finest lmc topology on $A$. Then, for any $\sum A^{2d}$-module $S$ of $A$,  $\cl{\eta}{S}=\pos{\K{S}}$. In particular, $\cl{\eta}{\sum A^{2d}}=\pos{\V{A}}$
\end{crl}
\begin{proof}
Apply Theorem \ref{MainThm2} with $\tau = \eta$, using the fact that $\M{\eta} = \V{A}$.
\end{proof}

\end{document}